\documentclass[12pt]{article}

\usepackage{amsmath,amsthm,amssymb,graphicx,pb-diagram,geometry,mathrsfs}
\usepackage{graphicx,epstopdf}
\usepackage{amsmath,latexsym,amssymb,amsthm}
\usepackage{sectsty}
\usepackage[lofdepth,lotdepth]{subfig}
\usepackage{array}
\usepackage{hyperref}
\usepackage{mathtools}
\usepackage{blkarray, bigstrut}
\usepackage{enumerate}   
\usepackage{adjustbox}
\usepackage{tikz}
\usepackage{amsfonts}

\theoremstyle{plain}
\newtheorem{theorem}{Theorem}
\newtheorem{lemma}[theorem]{Lemma}
\newtheorem{corollary}[theorem]{Corollary}

\theoremstyle{definition}
\newtheorem{definition}[theorem]{Definition}

\newtheorem{problem}[theorem]{Problem}

\theoremstyle{remark}

\begin{document}

\title{A diagram associated with the subconstituent algebra of a distance-regular graph}
\author{Supalak Sumalroj}
\author{Supalak  Sumalroj\\
	\small Department of Mathematics, Silpakorn University, 
	Nakhon Pathom, Thailand\\
	\small sumalroj\_s@silpakorn.edu
}
\date{}
\maketitle

\begin{abstract}
In this paper we consider a distance-regular graph $\Gamma$.
Fix a vertex $x$ of $\Gamma$ and consider the corresponding subconstituent algebra $T$.
The algebra $T$  is the $\mathbb{C}$-algebra generated by the Bose-Mesner algebra $M$ of $\Gamma$ and the dual Bose-Mesner algebra $M^*$ of $\Gamma$ with respect to $x$.
We consider the subspaces $M, M^*, MM^*, M^*M, MM^*M, M^*MM^*, \dots$ along with their intersections and sums.
In our notation, $MM^*$ means $Span\{RS|R\in M, S\in M^*\}$,  and so on.
We introduce a diagram that describes how these subspaces are related.
We describe in detail that part of the diagram up to $MM^*+M^*M$.
For each subspace $U$ shown in this part of the diagram, we display an orthogonal basis for $U$ along with the dimension of $U$.
For an edge $U\subseteq W$ from this part of the diagram, we display an orthogonal basis for the orthogonal complement of $U$ in $W$ along with the dimension of this orthogonal complement.\\


\noindent Math. Subj. Class.: 05E30  
\end{abstract}


\section{Introduction}

In this paper we consider a distance-regular graph $\Gamma$.
Fix a vertex $x$ of $\Gamma$ and consider the corresponding subconstituent algebra (or Terwilliger algebra) $T$ \cite{T_subconstituentI}.
The algebra $T$  is the $\mathbb{C}$-algebra generated by the Bose-Mesner algebra $M$ of $\Gamma$ and the dual Bose-Mesner algebra $M^*$ of $\Gamma$ with respect to $x$.
The algebra $T$ is finite-dimensional and  semisimple \cite{T_subconstituentI}.
So it is natural to compute the irreducible $T$-modules.
These modules are important in the study of 
hypercubes \cite{Go-hypercube,Miklavic_hypercubes},
dual polar graphs \cite{Lee_Tanaka,Worawannotai},
spin models \cite{Caughman_Wolff,Curtin_Nomura},
codes \cite{Gijswijt_Schrijver_Tanaka,Schrijver},
the bipartite property \cite{Caughman-bipartite-Qpoly,Caughman_1,Curtin,Lang,MacLean_Miklavic_endpt2,MacLean_Miklavic,MacLean_Miklavic_Penjic,MacLean_Terwilliger,Miklavic_Terwilliger},
the almost-bipartite property \cite{Caughman_MacLean_Terwilliger,Collins,Lang_Terwilliger}, 
the $Q$-polynomial property \cite{Caughman-bipartite-Qpoly,Cerzo,Dickie_Terwilliger,Gavrilyuk_Koolen,Lee,Lee_Q-poly,Miklavic_Terwilliger,Terwilliger-displacement}, and
the thin property \cite{Go_Terwilliger,MacLean_Terwilliger-Taut,Suzuki,Tanaka_Assmus,Terwilliger_inequality,Terwilliger-thin modules,Terwilliger_Weng,Terwilliger_Zitnik}.


In this paper we discuss the algebra $T$ using a difference approach. 
We consider the subspaces $M, M^*, MM^*, M^*M, MM^*M, M^*MM^*, \dots$ along with their intersections and sums; see Figure \ref{diagram}.
We describe the diagram of Figure \ref{diagram} up to $MM^*+M^*M$.
For each subspace $U$ shown in this part of the diagram, we display an orthogonal basis for $U$ along with the dimension of $U$.
For an edge $U\subseteq W$ from this part of the diagram, we display an orthogonal basis for the orthogonal complement of $U$ in $W$ along with the dimension of this orthogonal complement.
Our main results are summarized in Theorems \ref{orthog basis and dim}, \ref{orthog basis and dim of orthog complement}. 
In the last part of the paper we summarize what is known about the part of diagram above $MM^*+M^*M$, and we give some open problems.






\section{Preliminaries}

Let $X$ denote a nonempty finite set. 
Let ${{\rm{Mat}}}_X(\mathbb{C})$ denote the $\mathbb{C}$-algebra consisting of the matrices whose rows and columns are indexed by $X$ and whose entries are in $\mathbb{C}$. 
For $B\in{{\rm{Mat}}_X(\mathbb{C})}$ let $\overline{B}$, $B^t$, and $tr(B)$ denote the complex conjugate, the transpose, and the trace of $B$, respectively.
We endow ${\rm{Mat}}_X(\mathbb{C})$ with the Hermitean inner product $\langle \;,\;\rangle$  such that  $\langle R,S\rangle =tr(R^t\overline{S})$ for all $R, S\in{{\rm{Mat}}_X(\mathbb{C})}$.
The inner product $\langle \;,\;\rangle$ is positive definite.
Let $U, V$ denote subspaces of ${\rm{Mat}}_X(\mathbb{C})$ such that $U\subseteq V$.
The \textit{orthogonal complement} of $U$ in $V$ is defined by $U^\perp = \{v\in V|\langle v,u \rangle =0 \text{ for all } u\in U \}$.

Let $\Gamma=(X,\mathcal{E})$ denote a finite, undirected, connected graph, without loops or multiple edges, with vertex set $X$ and edge set $\mathcal{E}$.
Let $\partial$ denote the shortest path-length distance function for $\Gamma$.
Define the diameter $D:=\text{max}\{\partial(x,y)| x, y\in{X}\}$. 
For a vertex $x\in{X}$ and an integer $i\geq0$ define
$\Gamma_i(x)=\{y\in{X}|\partial(x,y)=i\}$.
For notational convenience abbreviate $\Gamma(x)=\Gamma_1(x)$.
For an integer $k\geq 0$, we say that $\Gamma$ is \textit{regular with valency $k$} whenever $|\Gamma(x)|=k$ for all $x\in{X}$.
We say that $\Gamma$ is \textit{distance-regular} whenever for all integers $h, i, j$ $(0\leq h, i, j\leq D)$ and $x, y \in{X}$ with $\partial(x, y) = h$, the number
$$p^h_{ij}:=|\Gamma_i(x)\cap\Gamma_j(y)|$$
is independent of $x$ and $y$.
The integers $p^h_{ij}$ are called the \textit{intersection numbers} of $\Gamma$.
From now on assume that $\Gamma$ is distance-regular with diameter $D\geq 3$.
We abbreviate $k_i:=p^0_{ii}$ $(0\leq i \leq D)$.
For $0\leq i \leq D$ let $A_i$ denote the matrix in ${\rm{Mat}}_X(\mathbb{C})$ with $(x,y)$-entry
\[ (A_i)_{xy}=
\begin{cases}
1       & \quad \text{if } \quad \partial(x,y)=i,\\
0  & \quad \text{if } \quad \partial(x,y)\neq i,\\
\end{cases}
\qquad \qquad x,y\in{X}.
\]
We call $A_i$ the \textit{$i$-th distance matrix} of $\Gamma$.
We call $A=A_1$ the \textit{adjacency matrix}  of $\Gamma$. 
Observe that $A_i$ is real and symmetric for $0\leq i \leq D$. 
Note that $A_0 = I$ is the identity matrix in ${\rm{Mat}}_X(\mathbb{C})$.
Observe that $\sum_{i=0}^{D} A_i=J$, where $J$ is the all-ones matrix in ${\rm{Mat}}_X(\mathbb{C})$.
Observe that for $0 \leq i, j \leq D$, 
\begin{align} \label{A_iA_j}
A_iA_j =\sum_{h=0}^{D} p^h_{ij}A_h.
\end{align}

\noindent For integers $h,i,j$ $(0\leq h,i,j \leq D)$ we have 
\begin{align} 
&p^h_{0j}=\delta_{hj} \label{p^h_{0j}} \\
&p^0_{ij}=\delta_{ij}k_i \label{p^0_{ij}}  
\end{align}


Let $M$ denote the subalgebra of ${\rm{Mat}}_X(\mathbb{C})$ generated by $A$. 
By \cite[p. 44]{BCN} the matrices $A_0,A_1,...,A_D$ form a basis for $M$.
We call $M$ the \textit{Bose-Mesner algebra of} $\Gamma$. 
By \cite[p. 59, 64]{BI_AlgCombI}, $M$ has a basis $E_0,E_1,...,E_D$ such that 
(i) $E_0 = |X|^{-1}J$; 
(ii) $\sum_{i=0}^{D} E_i=I$;
(iii) $E_i^t=E_i$ $(0\leq i \leq D)$;
(iv) $\overline{E_i}=E_i$ $(0\leq i \leq D)$;
(v) $E_iE_j = \delta_{ij}E_i$ $(0\leq i,j \leq D)$.
The matrices $E_0, E_1,...,E_D$ are called the \textit{primitive idempotents} of $\Gamma$, and $E_0$ is called the \textit{trivial} idempotent.
%
%
For $0\leq i \leq D$ let $m_i$ denote the rank of $E_i$.
%
For $0\leq i \leq D$ let $\theta_i$ denote an eigenvalue of $A$ associated with $E_i$.
\noindent Let $\lambda$ denote an indeterminate.
%
%
%
%
Define polynomials $\{u_i\}_{i=0}^{D}$ in $\mathbb{C}[\lambda]$ by 
$u_0=1$, $u_1=\lambda/k$, and
\begin{align*}
\lambda u_i=c_{i}u_{i-1}+a_iu_i+b_{i}u_{i+1} \qquad (1\leq i\leq D-1).
\end{align*}

\noindent By \cite[p. 131, 132]{BCN}, 
\begin{align} 
&A_j=k_j\displaystyle\sum_{i=0}^{D}u_j(\theta_i)E_i  &(0\leq j \leq D), \label{A_j}\\ 
&E_j=|X|^{-1}m_j\displaystyle\sum_{i=0}^{D} u_i(\theta_j)A_i  &(0\leq j \leq D). \label{E_j}
\end{align}

\noindent Since $E_iE_j = \delta_{ij}E_i$ and by (\ref{A_j}), (\ref{E_j}) we have $A_jE_i =k_ju_j(\theta_i)E_i= E_iA_j$ $(0 \leq i,j \leq D)$.
By \cite[Theorem 3.5]{BI_AlgCombI} we have the orthogonality relations
\begin{align}
&\displaystyle\sum_{i=0}^{D}u_i(\theta_r)u_i(\theta_s)k_i=\delta_{rs}m^{-1}_r|X| &(0\leq r,s\leq D), \label{sum u_i} \\
&\displaystyle\sum_{r=0}^{D}u_i(\theta_r)u_j(\theta_r)m_r=\delta_{ij}k^{-1}_i|X|   &  (0\leq i,j\leq D). \label{sum u_i_1}
\end{align}

We recall the Krein parameters of $\Gamma$.
Let $\circ$ denote the entry-wise multiplication in ${\rm{Mat}}_X(\mathbb{C})$.
Note that $A_i\circ A_j = \delta_{ij}A_i$ for $0 \leq i, j \leq D$.
So $M$ is closed under $\circ$. 
By \cite[p. 48]{BCN}, there exist scalars $q^h_{ij}\in{\mathbb{C}}$ such that
\begin{equation}\label{E_i_circE_j} 
\begin{split}
E_i\circ E_j=|X|^{-1}\displaystyle\sum_{h=0}^{D} q^h_{ij}E_h  \quad \qquad (0\leq i,j \leq D).
\end{split}
\end{equation}

\noindent We call the $q^h_{ij}$ the \textit{Krein parameters} of $\Gamma$. 
By \cite[Proposition 4.1.5]{BCN}, these parameters are real and nonnegative for $0\leq h,i,j\leq D$.

We recall the dual Bose-Mesner algebra of $\Gamma$.
Fix a vertex $x\in{X}$. 
For $0 \leq i \leq D$ let $E_i^*=E_i^*(x)$ denote the diagonal matrix in ${\rm{Mat}}_X(\mathbb{C})$ with $(y,y)$-entry
\[ (E_i^*)_{yy}=
\begin{cases}
1       & \quad \text{if } \quad \partial(x,y)=i,\\
0  & \quad \text{if } \quad \partial(x,y)\neq i,\\
\end{cases}
\qquad \qquad y\in{X}.
\]
We call $E_i^*$ the \textit{$i$-th dual idempotent} of $\Gamma$ with respect to $x$. 
Observe that (i) $\sum_{i=0}^{D} E_i^*=I$;
(ii)  $E_i^{*t}=E_i^*$ $(0\leq i \leq D)$;
(iii) $\overline{E_i^*}=E_i^*$ $(0\leq i \leq D)$; 
(iv) $E_i^*E_j^*=\delta_{ij}E_i^*$ $(0\leq i,j \leq D)$.
By construction $E_0^*,E_1^*,...,E_D^*$ are linearly independent.
Let $M^*=M^*(x)$ denote the subalgebra of ${\rm{Mat}}_X(\mathbb{C})$ with  basis $E_0^*,E_1^*,...,E_D^*$.
We call $M^*$ the \textit{dual Bose-Mesner algebra} of $\Gamma$ with respect
to $x$.

We now recall the dual distance matrices of $\Gamma$.
For $0 \leq i \leq D$ let $A_i^*=A_i^*(x)$ denote the diagonal matrix in ${\rm{Mat}}_X(\mathbb{C})$ with $(y,y)$-entry
\begin{equation}\label{A_i^*} 
\begin{split}
(A_i^*)_{yy}=|X|(E_i)_{xy} \qquad \qquad \qquad y\in{X}.
\end{split}
\end{equation} 
We call $A_i^*$ the \textit{dual distance matrix} of $\Gamma$ with respect to $x$ and $E_i$.
By \cite [p. 379]{T_subconstituentI}, the matrices $A_0^*,A_1^*,...,A_D^*$ 
form a basis for $M^*$.
Observe that
(i) $A_0^* = I$; 
(ii) $\sum_{i=0}^{D} A_i^*=|X|E_0^*$;
(iii) $A_i^{*t}=A_i^*$ $(0\leq i \leq D)$;
(iv) $\overline{A_i^*}=A_i^*$ $(0\leq i \leq D)$;
(v) $A_i^*A_j^* = \sum_{h=0}^{D} q^h_{ij}A_h^*$ $(0\leq i,j \leq D)$.
From (\ref{A_j}), (\ref{E_j}) we have
\begin{align} 
&A_j^*=m_j\displaystyle\sum_{i=0}^{D}u_i(\theta_j)E_i^*  &(0\leq j \leq D), \label{A_j^*} \\
&E_j^*=|X|^{-1}k_j\displaystyle\sum_{i=0}^{D} u_j(\theta_i)A_i^*  &(0\leq j \leq D). \label{E_j^*} 
\end{align}

\section{The subconstituent algebra $T$}

Let $T$ denote the subalgebra of ${\rm{Mat}}_X(\mathbb{C})$ generated by $M,M^*$.
The algebra $T$ is called the \textit{subconstituent algebra} (or
\textit{Terwilliger algebra}) \cite{T_subconstituentI}.
In order to describe $T$, we consider how $M,M^*$ are related.
We will use the following notation.
For any two subspaces $\mathcal{R},\mathcal{S}$ of  ${\rm{Mat}}_X(\mathbb{C})$ we define $\mathcal{R} \mathcal{S}=Span\{RS|R\in \mathcal{R}, S\in \mathcal{S}\}$. 
Consider the subspaces $M, M^*, MM^*, M^*M, MM^*M, M^*MM^*, \dots$ along with their intersections and sums.
To describe the inclusions among the resulting subspaces we draw a diagram; see Figure \ref{diagram}.
In this diagram, a line segment that goes upward from $U$ to $W$ means that  $W$ contains $U$.

\newpage
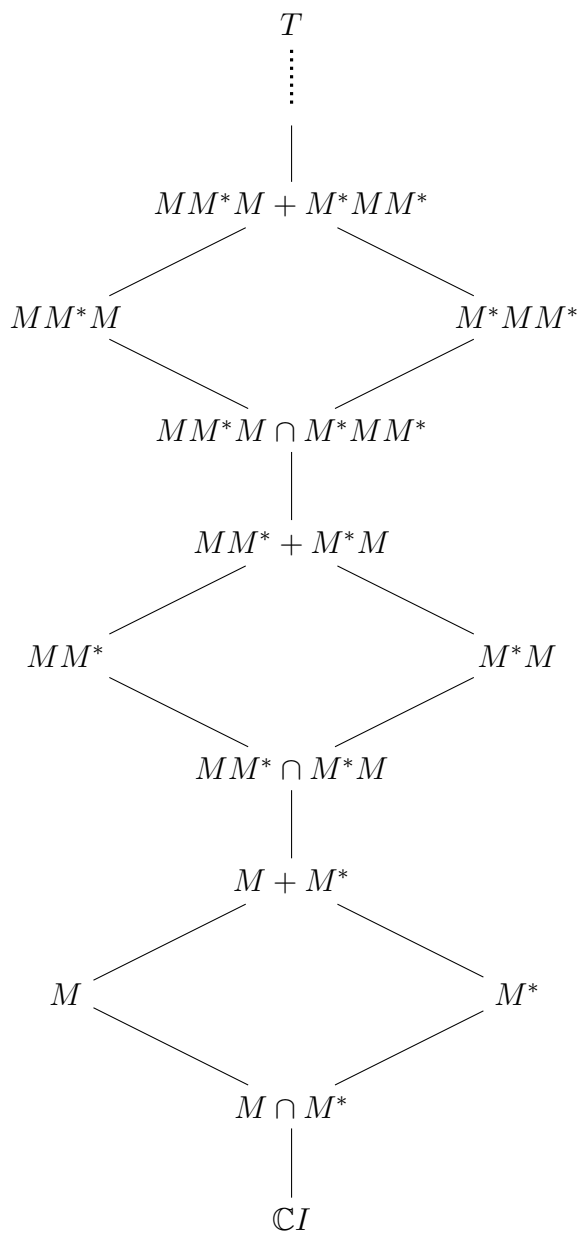
\begin{figure} [!]
\centering
\begin{tikzpicture}
\node (1) at (3,3) {$M^*MM^*$};
\node (2) at (-3,3) {$MM^*M$};
\node (3) at (0,1.5) {$MM^*M\cap M^*MM^*$};
\node (4) at (0,4.5) {$MM^*M + M^*MM^*$};
\node (5) at (0,5.7) {};
\node (6) at (0,6.9) {$T$};
\node (7) at (0,-9) {$\mathbb{C}I$};
\node (00) at (0,0) {$MM^*+M^*M$};
\node (11) at (3,-1.5)  {$M^*M$};
\node (12) at (-3,-1.5)  {$MM^*$}; 
\node (20) at (0,-3) {$MM^*\cap M^*M$};
\node (30) at (0,-4.5) {$M+M^*$};
\node (41) at (3,-6)  {$M^*$};
\node (42) at (-3,-6)  {$M$};
\node (50) at (0,-7.5) {$M\cap M^*$};
\draw (00) -- (11);
\draw (00) -- (12);
\draw (11) -- (20);
\draw (12) -- (20);
\draw (20) -- (30);
\draw (30) -- (41);
\draw (30) -- (42);
\draw (41) -- (50);
\draw (42) -- (50);
\draw (3) -- (00);
\draw (3) -- (1);
\draw (3) -- (2);
\draw (4) -- (1);
\draw (4) -- (2);
\draw (4) -- (5);
\draw (5) edge [very thick, dotted] (6);
\draw (50) -- (7);
\end{tikzpicture}
\caption{Inclusion diagram} \label{diagram}
\end{figure}

Consider the above diagram.
For each subspace $U$ shown in the diagram, we desire to find an orthogonal basis for $U$ and the dimension of $U$.
Also, for each edge $U\subseteq W$ shown in the diagram, we desire to find an orthogonal basis for the orthogonal compliment of $U$ in $W$ along with the dimension of this orthogonal compliment.
We accomplish these goals for that part of the diagram up to $MM^*+M^*M$.
Our main results are summarized in Theorems \ref{orthog basis and dim}, \ref{orthog basis and dim of orthog complement}.
Before we get started, we recall a few inner product formulas.

\begin{lemma}\label{formula} $\rm{(See ~\cite[Lemma~ 3.2]{T_subconstituentI}.)}$ For $0\leq h, i, j, r, s, t \leq D$,
	\begin{enumerate}[(i)]
		\item $\langle E_i^*A_jE_h^*,E_r^*A_sE_t^*\rangle =\delta_{ir}\delta_{js}\delta_{ht}k_hp^h_{ij}$,
		\item $\langle E_iA_j^*E_h,E_rA_s^*E_t\rangle =\delta_{ir}\delta_{js}\delta_{ht}m_hq^h_{ij}$.
	\end{enumerate}
\end{lemma}

The following result is well-known.

\begin{corollary} \label{well-known}
	For $0\leq h,i,j \leq D$,
		\begin{enumerate}[(i)]
		\item $E_i^*A_hE_j^*=0$ if and only if $p^h_{ij}=0$,
		\item $E_iA_h^*E_j=0$ if and only if $q^h_{ij}=0$.
	\end{enumerate}
\end{corollary}

\begin{lemma} \label{<>formula}  $\rm{(See ~\cite[Lemma~ 10]{Sumalroj}.)}$
	For $0\leq h, i, j, r, s, t \leq D$,
	\begin{equation*}
	\langle A_iE_j^*A_h,A_rE_s^*A_t\rangle =\displaystyle\sum_{\ell=0}^{D} k_\ell p^\ell_{ir}p^\ell_{js}p^\ell_{ht}.
	\end{equation*}
\end{lemma}


\section{The subspace $M+M^*$}

Our goal in this section is to analyze the inclusion diagram up to $M+M^*$.

%
%

\begin{lemma} \label{trace}
	For $0\leq i \leq D$,
	\begin{enumerate}[(i)]
		\item $tr(A_i)=\delta_{0i}|X|$,
		\item $tr(E_i)=m_i$,
		\item $tr(A^*_i)=\delta_{0i}|X|$,
		\item $tr(E^*_i)=k_i$.
	\end{enumerate}
\end{lemma}
\begin{proof}
	$(i)$ Follows from the definition of $A_i$.\\
	\noindent $(ii)$ Since $E_i$ is diagonalizable, we have $tr(E_i)=rank(E_i)=m_i$.\\
	\noindent $(iii)$  Follows from the definition of $A^*_i$.\\
	\noindent $(iv)$  Follows from the definition of $E^*_i$.
\end{proof}

\begin{lemma} \label{orthogonal} 
	For $0\leq  i, j \leq D$,
	\begin{enumerate}[(i)]
		\item $\langle A_i,A_j \rangle =\delta_{ij}k_i|X|$,
		\item $\langle E_i,E_j\rangle =\delta_{ij}m_i$,
		\item $\langle A^*_i,A^*_j \rangle =\delta_{ij}m_i|X|$,
		\item $\langle E^*_i,E^*_j\rangle =\delta_{ij}k_i$.
	\end{enumerate}
\end{lemma}
\begin{proof}
	$(i)$ Use (\ref{A_iA_j}) and Lemma \ref{trace}.\\
	$(ii)$ By Lemma \ref{trace} and since $E_iE_j=\delta_{ij}E_i$.\\
	$(iii),(iv)$ Similar to the proofs of $(i), (ii)$.
\end{proof}

\begin{lemma} \label{orthog basis M,M^*}
	Each of the following is an orthogonal basis for $M$:
	$$\{A_i\}^{D}_{i=0}, \qquad  \qquad \{E_i\}^{D}_{i=0}.$$
	Moreover, each of the following is an orthogonal basis for $M^*$:
	$$\{A^*_i\}^{D}_{i=0}, \qquad  \qquad \{E^*_i\}^{D}_{i=0}.$$
\end{lemma}
\begin{proof}
	Immediate from Lemma \ref{orthogonal}.
\end{proof}


\begin{lemma} \label{< A_i,A^*_j>}
	For $0\leq i,j\leq D$, $\langle A_i,A^*_j \rangle =\delta_{i0}\delta_{0j}|X|k_i$.
\end{lemma}
\begin{proof}
	Observe that 
	$\langle A_i,A^*_j \rangle =\langle A_iA^*_0A_0,A_0A^*_jA_0 \rangle$.
	By Lemma \ref{<>formula} and (\ref{sum u_i}), (\ref{A_j^*}), the result follows.
\end{proof}

Recall that $A_0=I=A_0^*$.

\begin{lemma} \label{basis M+M*}
	The following is an orthogonal basis for $M+M^*$:
	$$A_D,\dots,A_1,I,A^*_1\dots,A^*_D.$$
\end{lemma}
\begin{proof}
	Immediate from Lemmas \ref{orthogonal} and \ref{< A_i,A^*_j>}.
\end{proof}

\begin{lemma} \label{dim(M+M*)}
	$dim(M+M^*)=2D+1$.
\end{lemma}
\begin{proof}
	Immediate from Lemma \ref{basis M+M*}.
\end{proof}

\begin{lemma} \label{MintersectM^*}
	We have $M\cap M^*=\mathbb{C}I$ and $dim(M\cap M^*)=1$.
\end{lemma}
\begin{proof} 
	Observe that $I\in M\cap M^*$.
	By linear algebra, we have $dim(M\cap M^*)=dim(M)+dim(M^*)-dim(M+M^*)$.
	By construction $dim(M)=D+1$, $dim(M^*)=D+1$.
	By this and Lemma \ref{dim(M+M*)}, $dim(M\cap M^*)=1$.
	The result follows.
\end{proof}


\begin{lemma} \label{orthog coml}
	The following $(i)$--$(iv)$ hold:
	\begin{enumerate}[(i)]
		\item	The matrices $\{A_i\}_{i=1}^D$ form an orthogonal basis for the orthogonal complement of $M\cap M^*$ in $M$.
		\item  	The matrices $\{A^*_i\}_{i=1}^D$ form an orthogonal basis for the orthogonal complement of $M\cap M^*$ in $M^*$.
		\item  	The matrices $\{A_i\}_{i=1}^D$ form an orthogonal basis for the orthogonal complement of $M^*$ in $M+M^*$.
		\item  	The matrices $\{A^*_i\}_{i=1}^D$ form an orthogonal basis for the orthogonal complement of $M$ in $M+M^*$.
	\end{enumerate}
\end{lemma}
\begin{proof}
	Follows from definitions of $M,M^*$ along with Lemmas \ref{basis M+M*} and \ref{MintersectM^*}.
\end{proof}

\begin{lemma} \label{dim orthog coml}
	Each of the following subspaces has dimension $D$:
	\begin{align*}
	&(M\cap M^*)^\perp \cap M, \qquad \qquad &(M\cap M^*)^\perp \cap M^*,\\
	&(M^*)^\perp \cap (M+M^*),  &M^\perp \cap (M+M^*).
	\end{align*}
\end{lemma}
\begin{proof}
	Immediate from Lemma \ref{orthog coml}.
\end{proof}


\section{The subspace $MM^*+M^*M$}

Our goal in this section is to analyze the inclusion diagram from $M+M^*$ up to $MM^*+M^*M$.


\begin{lemma} \label{innerproduct}
	For $0\leq i,j,r,s \leq D$,
	\begin{enumerate}[(i)]
		\item 	$\langle A_iA^*_j,A^*_rA_s \rangle =\delta_{is}\delta_{jr}|X|k_im_ju_i(\theta_j)$,
		\item  $\langle A_iA^*_j,A_rA^*_s \rangle =\delta_{ir}\delta_{js}|X|k_im_j$.
	\end{enumerate}
\end{lemma} 
\begin{proof}
	$(i)$
	By (\ref{E_j}), (\ref{A_i^*}) and Lemma \ref{orthogonal}, we obtain
	\begin{align*} 
	\langle A_iA^*_j,A^*_rA_s \rangle 
	& = tr((A_iA_j^*)^t(\overline{A_r^*A_s})) \\
	& = tr(A_j^*A_iA_r^*A_s) 	\\
	& = \displaystyle\sum_{y\in X}\displaystyle\sum_{z\in X} (A_j^*)_{yy}(A_i)_{yz}(A_r^*)_{zz}(A_s)_{zy} \\
	& = |X|^2\displaystyle\sum_{y\in X}\displaystyle\sum_{z\in X} (E_j)_{xy}(A_i)_{yz}(E_r)_{xz}(A_s)_{zy} \\
	& = |X|\displaystyle\sum_{y\in X}\displaystyle\sum_{z\in X}\displaystyle\sum_{x\in X} (E_j)_{xy}(A_i)_{yz}(E_r)_{xz}(A_s)_{zy} \\
	& = |X|tr(E_jE_r(A_i \circ A_s))\\
	& = |X|tr((E_jE_r)^t(\overline{A_i \circ A_s}))\\
	& = |X|\langle E_jE_r,A_i \circ A_s \rangle\\
	& = \delta_{is}\delta_{jr}|X|\langle E_j,A_i\rangle\\
	& = \delta_{is}\delta_{jr}m_j\displaystyle\sum_{h=0}^{D}u_h(\theta_j)\langle A_h,A_i\rangle\\
	& = \delta_{is}\delta_{jr}m_j\displaystyle\sum_{h=0}^{D}u_h(\theta_j)\delta_{hi}k_i|X|\\
	& = \delta_{is}\delta_{jr}|X|k_im_ju_i(\theta_j). 
	\end{align*}
	
	$(ii)$
	By Lemma \ref{<>formula} and  (\ref{p^h_{0j}}), (\ref{p^0_{ij}}), (\ref{sum u_i}), (\ref{A_j^*}) we obtain
	\begin{align*} 
	\langle A_iA^*_j,A_rA^*_s \rangle 
	& = \langle A_iA^*_jA_0,A_rA^*_sA_0 \rangle \\
	& = m_jm_s\displaystyle\sum_{h=0}^{D}u_h(\theta_j)\displaystyle\sum_{\ell=0}^{D}u_\ell(\theta_s)\langle A_iE^*_hA_0,A_rE^*_\ell A_0 \rangle \\
	& = m_jm_s\displaystyle\sum_{h=0}^{D}u_h(\theta_j)\displaystyle\sum_{\ell=0}^{D}u_\ell(\theta_s)\displaystyle\sum_{t=0}^{D}k_tp^t_{ir}p^t_{h\ell}p^t_{00} \\
	& = m_jm_s\displaystyle\sum_{h=0}^{D}u_h(\theta_j)\displaystyle\sum_{\ell=0}^{D}u_\ell(\theta_s)k_0p^0_{ir}p^0_{h\ell} \\
	& = \delta_{ir}k_im_jm_s\displaystyle\sum_{h=0}^{D}u_h(\theta_j)u_h(\theta_s)k_h\\
	& = \delta_{ir}k_im_jm_s\delta_{js}m^{-1}_j|X|\\
	& = \delta_{ir}\delta_{js}|X|k_im_j. \qedhere
	\end{align*}
\end{proof}


\begin{lemma} \label{orthog basis MM^*,M^*M}
	The following $(i),(ii)$ hold:
	\begin{enumerate}[(i)]
	\item  The matrices $\{A_iA^*_j|0\leq i,j \leq D\}$ form an orthogonal basis for $MM^*$.
	\item  The matrices $\{A^*_jA_i|0\leq i,j \leq D\}$ form an orthogonal basis for $M^*M$.
	\end{enumerate}
\end{lemma}
\begin{proof}
	Immediate from Lemma \ref{innerproduct}.
\end{proof}

\begin{lemma} \label{dim orthog basis MM^*,M^*M}
	Each of the following subspaces has dimension $(D+1)^2$:
	\begin{align*}
	MM^*,  \qquad \qquad \qquad M^*M.
	\end{align*}
\end{lemma}
\begin{proof}
	Immediate from Lemma \ref{orthog basis MM^*,M^*M}. 
\end{proof}

\begin{lemma} \label{MM^*+M^*M}
	We have
	\begin{center}
		$MM^*+M^*M=\displaystyle\sum_{i=0}^{D}\displaystyle\sum_{j=0}^{D}Span(A_iA^*_j,A^*_jA_i)$ \qquad (orthogonal direct sum).
	\end{center}
\end{lemma}
\begin{proof}
	Immediate from Lemma \ref{innerproduct}.
\end{proof}

\begin{corollary} \label{dim(MM^*+M^*M)}
	We have 
	$$dim(MM^*+M^*M)=\displaystyle\sum_{i=0}^{D}\displaystyle\sum_{j=0}^{D}dim(Span(A_iA^*_j,A^*_jA_i)).$$
\end{corollary}
\begin{proof}
	Immediate from Lemma \ref{MM^*+M^*M}.
\end{proof}

\begin{definition} \label{H_ij}
	For $0\leq i,j \leq D$ let $H_{i,j}$ denote the $2\times2$ matrix of inner products for $A_iA^*_j, A^*_jA_i$.
\end{definition}

\begin{lemma} \label{matrix H}
	For $0\leq i,j \leq D$,
	\[
	H_{i,j}=|X|k_im_j\begin{pmatrix}
	  1  			& u_i(\theta_j)     	\\[1em]
	u_i(\theta_j)  	&  1     
	\end{pmatrix}.
	\]
	 
\end{lemma}
\begin{proof}
	Immediate from Lemma \ref{innerproduct} and Definition \ref{H_ij}.
\end{proof}


\begin{lemma} \label{det(H_ij)}
	For $0\leq i,j\leq D$ we have
	$$det(H_{i,j})=|X|^2k^2_im^2_j(1-(u_i(\theta_j))^2).$$
\end{lemma}
\begin{proof}
	Immediate from  Lemma \ref{matrix H}.
\end{proof}

\begin{corollary} \label{det(H_ij)=0}
	For $0\leq i,j\leq D$, 
	$det(H_{i,j})=0$ if and only if $u_i(\theta_j)=\pm1$.
\end{corollary}
\begin{proof}
	Immediate from  Lemma \ref{det(H_ij)}.
\end{proof}

\begin{lemma} \label{square norm}
	The following elements are orthogonal: 
	$$A_iA_j^*+A_j^*A_i,  \qquad A_iA_j^*-A_j^*A_i.$$
	Moreover 
	\begin{align*}
	||A_iA_j^*+A_j^*A_i||^2=2|X|k_im_j(1+u_i(\theta_j)),\\
	||A_iA_j^*-A_j^*A_i||^2=2|X|k_im_j(1-u_i(\theta_j)).
	\end{align*}
\end{lemma}
\begin{proof}
	Immediate from  Lemma \ref{matrix H}.
\end{proof}

\begin{lemma} \label{common values AiAj^*}
	The following $(i)$--$(iii)$ hold for $0\leq i,j\leq D$:
		\begin{enumerate}[(i)]
		\item  Assume $u_i(\theta_j)=1$. 
			Then $A_iA^*_j=A^*_jA_i$ and this common value is nonzero.
		\item  Assume $u_i(\theta_j)=-1$. 
			Then $A_iA^*_j=-A^*_jA_i$ and this common value is nonzero.
		\item  Assume $u_i(\theta_j)\neq \pm1$. 
			Then $A_iA^*_j,A^*_jA_i$ are linearly independent.
		\end{enumerate}
\end{lemma}
\begin{proof}
		$(i),(ii)$ 	Immediate from  Lemma \ref{square norm}.\\
		$(iii)$	Immediate from  Lemma \ref{det(H_ij)}.
	\end{proof}



\begin{lemma} \label{ortho u_i}
	For $0\leq i,j \leq D$ we give an orthogonal basis for $Span(A_iA^*_j,A^*_jA_i)$.
	\begin{center}
	\begin{tabular}{c|c|c}
		case		 & orthogonal basis 	& dimension \\ [.1em] \hline 
		$u_i(\theta_j)=\pm1$ 		& $A_iA^*_j$ 	& $1$\\ [.4em]
		$u_i(\theta_j)\neq\pm1$ 	&  $A_iA^*_j+A^*_jA_i,$ \; $A_iA^*_j-A^*_jA_i$	& $2$\\ [.4em] 
	\end{tabular}
	\end{center}
\end{lemma}
\begin{proof}
	Follows from Definition \ref{H_ij} and Lemmas \ref{matrix H}, \ref{common values AiAj^*}.
\end{proof}

\begin{corollary} \label{orthog basis MM^*+M^*M}
	The following is an orthogonal basis for $MM^*+M^*M$:
	$$\{A_iA^*_j+A^*_jA_i, A_iA^*_j-A^*_jA_i|0\leq i,j \leq D, u_i(\theta_j)\neq \pm1\}$$
	$$\cup\{A_iA^*_j|0\leq i,j \leq D, u_i(\theta_j)=\pm1\}.$$
\end{corollary}
\begin{proof}
	Immediate from Lemmas  \ref{MM^*+M^*M} and \ref{ortho u_i}.
\end{proof}

\begin{definition} \label{P}
	Define an integer $P$ as follows:
	$$P=|\{(i,j)|1\leq i,j \leq D, u_i(\theta_j)=\pm1\}|.$$
\end{definition}

\begin{lemma} \label{P=0}
	$P=0$ if and only if $\Gamma$ is primitive.
\end{lemma}
\begin{proof}
	Immediate from Definition \ref{P} and \cite [Proposition~4.4.7]{BCN}.
\end{proof}

\begin{lemma} \label{dim MM^* + M^*M}
	$dim(MM^*+M^*M)=2D^2+2D+1-P$.
\end{lemma}
\begin{proof}
	Immediate from Corollary \ref{orthog basis MM^*+M^*M} and Definition \ref{P}.
\end{proof}

\begin{lemma} \label{dim MM^* cap M^*M}
	$dim(MM^*\cap M^*M)=2D+1+P$.
\end{lemma}
\begin{proof}
	By linear algebra, we have 
	$$dim(MM^*\cap M^*M)=dim(MM^*)+dim(M^*M)-dim(MM^*+M^*M).$$
	By Lemmas \ref{dim orthog basis MM^*,M^*M} and \ref{dim MM^* + M^*M}, the result follows.
\end{proof}

\begin{lemma} \label{orthog basis MM^* cap M^*M}
	The following is an orthogonal basis for  $MM^*\cap M^*M$:
	$$\{A_iA^*_j|0\leq i,j \leq D, u_i(\theta_j)=\pm1\}.$$
\end{lemma}
\begin{proof}
	Immediate from Corollary \ref{det(H_ij)=0} and Lemma \ref{dim MM^* cap M^*M}.
\end{proof}

\begin{lemma} \label{orthog coml_2}
	The matrices $\{A_iA^*_j|1 \leq i,j \leq D, u_i(\theta_j)=\pm1\}$ form an orthogonal basis for the orthogonal complement of $M+M^*$ in $MM^*\cap M^*M$.
\end{lemma}
\begin{proof}
	Follows from Lemmas \ref{basis M+M*} and \ref{orthog basis MM^* cap M^*M}.
\end{proof}

\begin{lemma} \label{orthog coml_3}
	The following $(i),(ii)$ hold:
	\begin{enumerate}[(i)]
		\item  	The matrices $\{A_iA^*_j|1 \leq i,j \leq D, u_i(\theta_j)\neq \pm1\}$ form an orthogonal basis for the orthogonal complement of $MM^*\cap M^*M$ in $MM^*$.
		\item  	The matrices $\{A^*_jA_i|1 \leq i,j \leq D, u_i(\theta_j)\neq\pm1\}$ form an orthogonal basis for the orthogonal complement of $MM^*\cap M^*M$ in $M^*M$.
	\end{enumerate}
\end{lemma}
\begin{proof}
	Follows from Lemmas \ref{orthog basis MM^*,M^*M} and \ref{orthog basis MM^* cap M^*M}.
\end{proof}

\begin{lemma} \label{orthog coml_4}
	The following $(i),(ii)$ hold:  
	\begin{enumerate}[(i)]
		\item  	The matrices $\{u_i(\theta_j)A_iA^*_j-A^*_jA_i |1 \leq i,j \leq D, u_i(\theta_j)\neq \pm 1
		\}$ form an orthogonal basis for the orthogonal complement of $MM^*$ in $MM^*+M^*M$.
		\item  	The matrices $\{A_iA^*_j-u_i(\theta_j)A^*_jA_i |1 \leq i,j \leq D, u_i(\theta_j)\neq \pm 1
		\}$ form an orthogonal basis for the orthogonal complement of $M^*M$ in $MM^*+M^*M$.
	\end{enumerate}
\end{lemma}
\begin{proof}
	Follows from Lemma \ref{orthog basis MM^*,M^*M} and Corollary \ref{orthog basis MM^*+M^*M}.
\end{proof}

\begin{lemma} \label{dim orthog coml_2}
	The following subspace has dimension $P$:
	$$(M+M^*)^{\perp}\cap (MM^*\cap M^*M).$$
\end{lemma}
\begin{proof}
	Immediate from Definition \ref{P} and Lemma \ref{orthog coml_2}.
\end{proof}

\begin{lemma} \label{dim orthog coml_3}
	Each of the following subspaces has dimension $D^2-P$:
	\begin{align*}
	&(MM^*\cap M^*M)^{\perp}\cap MM^*, \qquad \quad &(MM^*\cap M^*M)^{\perp}\cap M^*M,\\
	&(MM^*)^{\perp}\cap (MM^*+ M^*M),  & (M^*M)^{\perp}\cap (MM^*+ M^*M).
	\end{align*} 
	\end{lemma}
\begin{proof}
	Immediate from Definition \ref{P} and Lemmas \ref{orthog coml_3}, \ref{orthog coml_4}.
\end{proof}


\section{Summary of main results}

We summarize our results in the following theorems.

\begin{theorem} \label{orthog basis and dim}
	In each row of the table below we describe a subspace $U$ in the diagram of Figure \ref{diagram}. 
	We give an orthogonal basis for $U$ along with the dimension of $U$.
	\begin{center}
	\renewcommand{\arraystretch}{1.3}
	\begin{tabular}{ c | c | c }  
		subspace $U$ & orthogonal basis for $U$ & dimension of $U$\\ 	[.1em]  
		\hline 	
		$M\cap M^*$ & $I$ 				& $1$ \\ [.4em] 
		
		$M$ 		& $\{A_i\}_{i=0}^D$ 	&$D+1$ \\ [.4em]
		
		$M^*$ 		& $\{A_i^*\}_{i=0}^D$ 	&$D+1$  \\ [.4em]
		
		$M+M^*$ 	& $\{A_D,\dots,A_1,I,A_1^*,\dots,A_D^*\}$ 									& $2D+1$\\ [.4em]

		$MM^*\cap M^*M$ & $\{A_iA^*_j|0\leq i,j \leq D, u_i(\theta_j)=\pm1\}$ & $2D+1+P$\\ [.4em]
		
		$MM^*$ & $\{A_iA^*_j|0\leq i,j \leq D\}$	& $(D+1)^2$ \\ [.4em]
		
		$M^*M$ & $\{A^*_jA_i|0\leq i,j \leq D\}$	& $(D+1)^2$ \\ [.4em]
		
		$MM^*+M^*M$ & $\{A_iA^*_j+A^*_jA_i,A_iA^*_j-A^*_jA_i$ & $2D^2+2D+1-P$ \\ [.4em] 
		&$|0\leq i,j \leq D, u_i(\theta_j)\neq \pm1\}$ & \\ [.4em]
		&$ \cup\{A_iA^*_j|0\leq i,j \leq D, u_i(\theta_j)=\pm1\}$ &\\ [.4em] 
	\end{tabular}
	\\ [1.5em]
\end{center}
\end{theorem}

\begin{theorem} \label{orthog basis and dim of orthog complement}
In each row of the table below we describe an edge $U\subseteq W$ from the diagram of Figure \ref{diagram}.
We give an orthogonal basis for the orthogonal complement of $U$ in $W$ along with the dimension of this orthogonal complement.
\begin{center}
	\renewcommand{\arraystretch}{1.3}
	\begin{tabular}{ c | c | c | c }  
		$U$ & $W$  & orthogonal basis   & dimension  \\
		 &   &  for $U^\perp \cap W$  &  of $U^\perp \cap W$ \\ [.1em]  
		\hline

		$M\cap M^*$ & $M$ & $\{A_i\}_{i=1}^D$ 	& $D$\\ [.4em]
		
		$M\cap M^*$ & $M^*$ & $\{A_i^*\}_{i=1}^D$ & $D$ \\ [.4em]
		
		$M$ & $M+M^*$ & $\{A_i^*\}_{i=1}^D$   & $D$\\ [.4em]
		
		$M^*$ & $M+M^*$ & $\{A_i\}_{i=1}^D$ & $D$\\ [.4em]
		
		$M+M^*$ & $MM^*\cap M^*M$ & $\{A_iA^*_j|1\leq i,j \leq D, u_i(\theta_j)=\pm1\}$	& $P$ \\ [.4em]
		
		$MM^*\cap M^*M$ & $MM^*$ & $\{A_iA^*_j|1\leq i,j \leq D, u_i(\theta_j)\neq \pm1\}$	& $D^2-P$ \\ [.4em]
		
		$MM^*\cap M^*M$ & $M^*M$ & $\{A^*_jA_i|1\leq i,j \leq D, u_i(\theta_j)\neq \pm1\}$ 	& $D^2-P$ \\ [.4em]
		
		$MM^*$ & $MM^*+ M^*M$ & $\{u_i(\theta_j)A_iA^*_j-A^*_jA_i|1 \leq i,j \leq D,$ & $D^2-P$ \\ [.4em] 
		& &$  u_i(\theta_j)\neq \pm 1\}$ & \\ [.4em]
		
		$M^*M$ & $MM^*+ M^*M$ & $\{A_iA^*_j-u_i(\theta_j)A^*_jA_i |1 \leq i,j \leq D, $ & $D^2-P$\\ [.4em]
		& &$ u_i(\theta_j)\neq \pm 1\}$ & \\ [.4em] 
		\end{tabular}
\end{center}
\end{theorem}


\section{Open problems}

In this section, we give some open problems and suggestions for future research.
Earlier in the paper we discussed the diagram of Figure \ref{diagram}.
In this discussion we analyzed the diagram up to $MM^*+M^*M$.
The remaining part of the diagram is not completely understood.
We mention what is known.
By Lemma \ref{formula} the subspace $M^*MM^*$ has an orthogonal basis $\{E_i^*A_jE_h^*|0\leq h,i,j \leq D,p^h_{ij}\neq 0\}$.
Similarly, the subspace $MM^*M$ has an orthogonal basis $\{E_iA_j^*E_h|0\leq h,i,j \leq D,q^h_{ij}\neq 0\}$.

\begin{problem}
	Find an orthogonal basis for the following subspaces:
	\begin{enumerate}[(i)]
		\item  $MM^*M\cap M^*MM^*,$
		\item  $MM^*M+M^*MM^*$.
	\end{enumerate}
\end{problem}

\begin{problem}
	In each row of the table below we give an edge $U\subseteq W$ from the diagram of Figure \ref{diagram}.
	Find an orthogonal basis for the orthogonal complement of $U$ in $W$.
	\begin{center}
		\begin{tabular}{c|c}
			$U$		 				& $W$	 \\ [.1em] \hline 
			$MM^*+M^*M$ 			& $MM^*M\cap M^*MM^*$ 	\\ [.4em]
			$MM^*M\cap M^*MM^*$ 	& $MM^*M$   \\ [.4em] 
			$MM^*M\cap M^*MM^*$		& $M^*MM^*$  \\ [.4em]
			$MM^*M$					& $MM^*M+M^*MM^*$ \\ [.4em]
			$M^*MM^*$				& $MM^*M+M^*MM^*$\\ [.4em]
		\end{tabular}
	\end{center}
\end{problem}


\section{Acknowledgement}

The author would like to thank Professor Paul Terwilliger for many valuable ideas and insightful suggestions on my work.
This paper was written while the author was an Honorary Fellow at the University of Wisconsin-Madison (January 2017 -- January 2018) supported by  the Development and Promotion of Science and Technology Talents (DPST)
Project, Thailand.


\end{document}